\documentclass[a4paper,12pt]{article}

\usepackage{amsmath ,amsfonts,  amssymb, amscd, amsthm} 
\usepackage{bbm}
\usepackage{mathrsfs}
\usepackage[font=small,skip=0pt]{caption}

\usepackage{csquotes, color}
\usepackage{colortbl}
\usepackage[numbers]{natbib}
\usepackage{verbatim}
\usepackage[toc,page]{appendix}
\usepackage{caption}
\usepackage{textcomp}
\usepackage[inline]{enumitem}
\usepackage{tcolorbox}
\usepackage{cancel}
\usepackage{relsize}
\usepackage{bm}

\usepackage{authblk}

\newtheorem{theorem}{Theorem}
\newtheorem{corollary}{Corollary}[theorem]

\newtheorem{proposition}{Proposition}

\newtheorem{definition}{Definition}
\newtheorem*{definition*}{Definition}

\newtheorem*{example*}{Example}

\newtheorem*{conditional*}{Conditional}

\newtheorem*{remark*}{Remark}

\definecolor{Gray}{gray}{0.85}
\definecolor{LightCyan}{rgb}{0.88,1,1}

\newcolumntype{a}{>{\columncolor{Gray}}c}
\newcolumntype{b}{>{\columncolor{Gray}}r}
\newcolumntype{d}{>{\columncolor{white}}c}

\newcommand{\cE}{{\cal E}}

\newcommand{\scrL}{{\mathscr L}}

\usepackage{hyperref}
\hypersetup{
    colorlinks=true,
    citecolor=black,
    linkcolor=black,
    filecolor=cyan,      
    urlcolor=blue,
}
 
\providecommand{\keywords}[1]
{
  \small	
  \textbf{\textit{Keywords---}} #1
}

\title{
Stable Densities, Fractional Integrals 
and the 
Mittag-Leffler Function 
}
\author{Nomvelo Karabo Sibisi \\{\small {\tt sbsnom005@myuct.ac.za}}}
\date{\today}

\begin{document}
\maketitle
\thispagestyle{empty}


\begin{abstract}
\noindent
This paper combines  probability  theory and fractional calculus to derive  a novel integral representation of the 
three-parameter Mittag-Leffler function or  Prabhakar function, 
where the  three parameters  are combinations of four base  parameters.
The fundamental concept  
is the Riemann-Liouville fractional integral of the one-sided  stable density, conditioned on a scale factor.
Integrating  with respect to a gamma  distributed scale factor 
induces a  mixture  of  Riemann-Liouville integrals. 
A particular combination of four base parameters leads to a representation of the Prabhakar function
as a weighted mixture of Riemann-Liouville  integrals at different scales.
The Prabhakar function  constructed in this manner is the Laplace transform of a four-parameter   distribution.
This  general approach gives various  known results  as special cases (notably,  the two-parameter generalised Mittag-Leffler distribution).
\end{abstract}

\section{Introduction}
\label{sec:intro}

There is an   intimate relationship between Mittag-Leffler  functions and  fractional calculus,  
as comprehensively discussed in   Gorenflo {\it et al.}~\cite{gorenflo2014mittag}. 
The three-parameter Mittag-Leffler function known as   the  Prabhakar function has such a  prominent role in this context  as to inspire the phrase
 ``Prabhakar fractional calculus''  (Giusti {\it et al.}~\cite{Giusti}).
At its core,  this involves fractional integrals and derivatives of functions involving the Prabhakar function. 

This paper constructs the Prabhakar function itself from fractional integration combined with probability theory. 
The point of departure is 
the density of the one-sided stable distribution, conditioned on a scale parameter,
to which we assign a gamma distribution.
Integrating over the scale parameter  leads to a novel   integral representation  
of the Prabhakar function as a weighted mixture of Riemann-Liouville  integrals of stable densities.
The formal statement of this construction  is contained  in Theorem~\ref{thm:ML4par}. 
The three parameters of the  Prabhakar function involved in the theorem are, in turn, constructed as combinations of four other parameters
that we refer to as base parameters.

Theorem~\ref{thm:ML4par} has an interesting corollary -- 
the integral representation readily leads to a proof that the Prabhakar function constructed  from the four parameters is completely monotone,
{\it i.e.}\ that the Prabhakar function is the Laplace  transform of a distribution  parameterised by the four base parameters.
This proof arises from a probabilistic foundation, as opposed to the complex analytic approaches that are normally adopted in the exploration
of the completely monotone character  of Mittag-Leffler functions.
Accordingly, we shall devote a good proportion of this paper to a discussion of the corollary on complete monotonicity.

The study of the complete monotonicity of the Mittag-Leffler function
dates back at least to  a 1948 paper by Pollard~\cite{PollardML},
who  used  a complex analytic method to prove the property for the 
one-parameter Mittag-Leffler function (Gorenflo {\it et al.}~\cite{gorenflo2014mittag} (3.7.2) reproduces Pollard's argument).
G\'{o}rska {\it  et al.}~\cite{Gorska} adapted  Pollard's complex analytic method to prove the complete monotonicity of  the 
three-parameter Mittag-Leffler function.

The Mittag-Leffler function also has an intimate connection with probability theory.
In his proof, 
Pollard noted that the ultimate result eluded him ``without the intervention of  
[the  Laplace transform of the stable distribution]".
While  he did not  use probabilistic   language, 
Pollard cited personal communication by Feller of a discovery of the result by  ``methods of probability theory''.

The completely monotone character of the Mittag-Leffler function is  much more than an abstract mathematical curiosity.
It plays a  fundamental role in models of  physical  phenomena 
such as anomalous dielectric relaxation and viscoelasticity as mentioned, for example, by 
de Oliviera {\it  et al.}~\cite{Oliveira}, Garra and Garrappa~\cite{Garra},
G\'{o}rska {\it  et al.}~\cite{Gorska}, Mainardi and Garrappa~\cite{MainardiGarrappa}.

The conceptual underpinning  of this paper, in the form of probability theory and fractional calculus,
bears similarity to the work of Ho {\it  et al.}~\cite{HoJamesLau} but there is a difference in context and purpose.
The   latter paper is on random partitions of the integers, with a different discussion of the Prabhakar function 
from that forming the central theme of this paper.

We start by introducing  the primary concepts in Section~\ref{sec:primary} required for our main contribution in Section~\ref{sec:main}.
This is followed by a discussion  of various known  distributions  arising  as special cases of Theorem~\ref{thm:ML4par},
notably the generalised Mittag-Leffler distribution.

\section{Primary Concepts}
\label{sec:primary}

\subsection{Distributions}
\label{sec:stable}

\subsubsection{Stable Distribution}
\label{sec:distributions}

The one-sided  stable distribution $F_\alpha(x\vert t)$ ( $0<\alpha<1$)  on $x\ge0$, 
conditioned  on a  scale parameter $t>0$, is indirectly defined by   its Laplace-Stieltjes transform
$($equivalently, the ordinary Laplace transform of its density $f_\alpha(x\vert t)$)
\begin{align}
e^{-t s^\alpha} &= \int_0^\infty e^{-sx}\,dF_\alpha(x\vert t) =  \int_0^\infty e^{-s x}  f_\alpha(x\vert t)\, dx
\label{eq:stable} 
\end{align}
The explicit form of $F_\alpha$ will not concern us here.
In any case, it is only known in closed form for selected values of $\alpha$, the simplest  being for $\alpha=1/2$.
We  write $F_\alpha(x)\equiv F_\alpha(x\vert t=1)$ and $f_\alpha(x)\equiv f_\alpha(x\vert t=1)$. 
It follows that $F_\alpha(x\vert t) \equiv F_\alpha(xt^{-1/\alpha})$
and $f_\alpha(x\vert t) \equiv f_\alpha(x t^{-1/\alpha})\, t^{-1/\alpha}$.
We may include $\alpha=1$ by defining $f_{\alpha=1}(x\vert t)=\delta(x-t)$ ($\delta$ is the Dirac delta function) 
with Laplace transform $e^{-t s}$.

\subsubsection{Gamma Distribution}
\label{sec:gamma}

The  gamma distribution  $G(x\vert \mu,\lambda)$  ($x>0$), with shape and scale parameters  $\mu>0, \lambda\ge0$ respectively,
is given  by 
\begin{align}
dG(x \vert \mu,\lambda) 
 &=\frac{1}{\Gamma(\mu)}\,x^{\mu-1}e^{-\lambda x}\, dx 
\label{eq:gammadistribution} 
\end{align}
We have omitted the usual normalising  term $\lambda^\mu$ 
in order to accommodate the $\lambda=0$ case.

\subsection{Mittag-Leffler Function}
\label{sec:mitlef}

The (one-parameter) Mittag-Leffler function 
$E_\alpha(x)$  is defined by  the infinite series 
\begin{align}
E_\alpha(x) &= \sum_{k=0}^\infty \frac{x^k}{\Gamma(\alpha k+1)} \quad \alpha\ge0
\label{eq:ML}
\end{align}
The Laplace transform of 
$E_\alpha(-\lambda x^\alpha)$ $(x\ge0, \lambda\ge0)$ is 
\begin{align}
\int_0^\infty e^{-sx} E_\alpha(-\lambda x^\alpha) \,dx  &= \frac{s^{\alpha-1}}{\lambda+s^\alpha} \qquad {\rm Re}(s)\ge0
\label{eq:LaplaceML}
\end{align}

There is a  three-parameter generalisation  of the Mittag-Leffler function, also known as the Prabhakar function,  defined  by  
\begin{align}
E^\gamma_{\alpha,\beta}(x) &= \frac{1}{\Gamma(\gamma)} 
   \sum_{k=0}^\infty \frac{\Gamma(\gamma+k)}{k!\,\Gamma(\alpha k+\beta)}\, x^k  
\label{eq:ML3parseries}
\end{align}
The one-parameter Mittag-Leffler function $E_\alpha(x)$ is the special case $\gamma=\beta=1$. 
The Laplace transform of $\cE^\gamma_{\alpha,\beta}(x\vert \lambda)  \equiv x^{\beta-1}E^\gamma_{\alpha,\beta}(-\lambda x^\alpha)$  is
\begin{align}
\int_0^\infty e^{-sx} \, \cE^\gamma_{\alpha,\beta}(x\vert \lambda)  \,dx  
  &= \frac{s^{\alpha\gamma-\beta}}{(\lambda+s^\alpha)^\gamma} 
\label{eq:LaplaceML3par}
\end{align}
This is known as Prabhakar's result in other literature, while $\cE^\gamma_{\alpha,\beta}(x\vert \lambda)$ is known as the Prabhakar kernel.

\subsection{Fractional Integration}
\label{sec:fracint}

The right-sided Riemann-Liouville fractional integral of a function $f(x)$ on $x\ge0$ for $\nu>0$ is defined by
\begin{align}
\{I_+^\nu\, f\}(x) &\equiv \{h_\nu\star f\}(x) = \frac{1}{\Gamma(\nu)} \int^x_0 (x-u)^{\nu-1} f(u)\, du 
\label{eq:RL} 
\intertext{where $h_\nu(x)=x^{\nu-1}/\Gamma(\nu)$ and $h_\nu\star f$ denotes Laplace convolution.
By the convolution theorem, the Laplace transform of {\rm (\ref{eq:RL})} is} 
\scrL\{I_+^\nu\, f\}(s) &\equiv \scrL\{h_\nu\star f\}(s) =   \widetilde h_\nu(s)\times\widetilde f(s) = s^{-\nu}\widetilde f(s) 
\label{eq:RLLaplace}
\end{align}
where $\scrL\{f\}$ and  $\widetilde f$ both denote the Laplace  transform of $f$.
Since (\ref{eq:RLLaplace})  reduces to 
$\widetilde f(s)$   for $\nu=0$,
we may  define the 
fractional integral for $\nu=0$ as  $\{I_+^{\,0}\, f\}(x)\equiv \{h_0\star f\}(x)=  f(x)$ 
where $h_0(x)=\delta(x)$.
We shall continue to  use the form~(\ref{eq:RL})  for   $\{I_+^{\nu}\, f\}(x)$ ($\nu\ge0$), with the implicit  understanding that, for $\nu=0$,
$\{I_+^{\,0}\, f\}(x) = f(x)$.

This definition of the Riemann-Liouville fractional integral, equivalent to Laplace convolution, suffices for our purposes.
In a more general definition of (\ref{eq:RL}), the integration lower limit  need not be zero,  
in which case the  Laplace convolution equivalence  no longer holds.
There is also a left-sided variant $I_-^\nu\, f$  that  is not relevant here.

We shall denote  the  Riemann-Liouville fractional integral of a distribution  $F(x)$ as
\begin{align}
\frac{1}{\Gamma(\nu)} \int^x_0 (x-u)^{\nu-1} dF(u) \quad \nu\ge0
\label{eq:RLdistribution} 
\end{align}
If $F(x)$ has a density $f(x)$, then (\ref{eq:RLdistribution}) and (\ref{eq:RL})  are equivalent.
This  holds for the stable distribution $F_\alpha(x\vert t )$, density $f_\alpha(x\vert t )$, whose fractional integral is our fundamental object of study.

The Riemann-Liouville fractional integral satisfies what is often referred to  as the  semigroup property 
$I_+^{\nu_1} I_+^{\nu_2} = I_+^{\nu_1+\nu_2}$.
This is equivalent to the convolution statement that $h_{\nu_1}\star h_{\nu_2} = h_{\nu_1+\nu_2}$ or the Laplace transform equivalent
$s^{-\nu_1}s^{-\nu_2}=s^{-(\nu_1+\nu_2)}$.

\subsection{Complete Monotonicity}
\label{sec:CM}

An infinitely differentiable function $\varphi(x)$ on $x>0$ is completely monotone  if its derivatives $\varphi^{(n)}(x)$ 
satisfy $(-1)^n\varphi^{(n)}(x)\ge0$, $n\ge 0$.
Bernstein's theorem 
states that $\varphi(x)$ is completely monotone iff it may be expressed as 
the Laplace-Stieltjes transform
\begin{align}
\varphi(x) &= \int_0^\infty e^{-x t}\,dF(t) 
\label{eq:LT} 
\end{align}
for  a non-decreasing distribution function $F(t)$ 
$($we shall refer to the  Laplace-Stieltjes transform merely as the Laplace transform, except when confusion might arise$)$.
For bounded $F(t)$, $\varphi(x)$ is defined on $x\ge0$.

\section{Main Contribution}
\label{sec:main}

The fundamental probabilistic construct of this paper, underpinning all else that follows, 
is the  Riemann-Liouville fractional integral of the conditional stable density $f_\alpha(x\vert t)$ ($0<\alpha\le1$):
\begin{alignat}{3}
\{I_+^\nu\, f_\alpha(\cdot\vert t)\}(x) &\equiv \{h_\nu\star f_\alpha(\cdot\vert t)\}(x) 
   &&= \frac{1}{\Gamma(\nu)} \int^x_0 (x-u)^{\nu-1} f_\alpha(u\vert t)\, du 
 \label{eq:RLstable} \\
& &&=  \frac{1}{\Gamma(\nu)} \int^x_0 (x-u)^{\nu-1} dF_\alpha(u\vert t) 
 \label{eq:RLstableDistribution} \\
\scrL\{I_+^\nu\, f_\alpha(\cdot\vert t)\}(s) 
&= s^{-\nu}\,e^{-t s^\alpha}
\label{eq:RLLaplacestableDistribution}
\end{alignat}
For completeness, $\{I_+^0 \, f_\alpha(\cdot\vert t)\}(x)=f_\alpha(x\vert t)$ and, in keeping  with   $f_{\alpha=1}(x\vert t)=\delta(x-t)$,
\begin{align}
\{I_+^\nu\, f_{\alpha=1}(\cdot\vert t)\}(x) 
&=
\begin{cases}
    \frac{1}{\Gamma(\nu)}  (x-t)^{\nu-1}  & t\le x \\
    0              & t>x 
\end{cases}
 \label{eq:RLstable1} 
\end{align}
$\{I_+^\nu\, f_\alpha(\cdot\vert t)\}(x)$, $\nu\ge0$,  is clearly nonnegative.
The key idea  is to assign a distribution $G(t)$, say,  to the scale factor $t$ and 
integrate with respect to $G(t)$ to generate a mixture density.
Rather than just  settling on the single value $t=1$, 
this takes into account all values of $t$, each weighted by the corresponding $G(t)$.
Although one might explore various choices of $G(t)$, 
our particular  interest  here is in the gamma distribution $G(t\vert \mu, \lambda)$ of~(\ref{eq:gammadistribution}) above.

Proposition~\ref{prop:RLstable} states a   property of $\{I_+^\nu\, f_\alpha(\cdot\vert t)\}(x)$ that 
is key to the  theorem that follows. 

\begin{proposition}
\label{prop:RLstable}
$\{I_+^\nu\, f_\alpha(\cdot\vert t)\}$ and $\{I_+^\nu\, f_\alpha\}\equiv\{I_+^\nu\, f_\alpha(\cdot\vert t=1)\}$, $\nu\ge0$,
are related by the   identity
\begin{align}
 \{I_+^\nu\, f_\alpha(\cdot\vert t)\}(x) &= t^{(\nu-1)/\alpha} \{I_+^\nu\, f_\alpha\}(xt^{-1/\alpha})
\label{eq:RLstable}
\end{align}
\end{proposition}

\begin{proof}[Proof of Proposition  \ref{prop:RLstable}]
\label{proof:RLstable}
$\{I_+^\nu\, f_\alpha(\cdot\vert t)\}(x)$  takes the explicit form:
\begin{align*} 
\{I_+^\nu\, f_\alpha(\cdot\vert t)\}(x)  &= \ \frac{1}{\Gamma(\nu)} \int^x_0 (x-u)^{\nu-1} f_\alpha(u\vert t)\, du  \\
&=  \frac{1}{\Gamma(\nu)} \int^x_0 (x-u)^{\nu-1} f_\alpha(u t^{-1/\alpha})\, t^{-1/\alpha} \, du  \\
y=ut^{-1/\alpha}:\quad
&=  \frac{1}{\Gamma(\nu)} \int^{xt^{-1/\alpha}}_0 (x-yt^{1/\alpha})^{\nu-1} f_\alpha(y) \, dy  \\
&=  \frac{ t^{(\nu-1)/\alpha}}{\Gamma(\nu)} \int^{xt^{-1/\alpha}}_0 (xt^{-1/\alpha}-y)^{\nu-1} f_\alpha(y) \, dy  
\end{align*}
The last expression is  the explicit form of  $t^{(\nu-1)/\alpha} \{I_+^\nu\, f_\alpha\}(xt^{-1/\alpha})$.
\end{proof}

\begin{definition}[Mixture]
\label{def:RLmixture}
Let  $M^\nu_{\alpha,\mu}(x\vert \lambda)$ $(x\ge0)$ 
be a mixture of fractional integrals $\{I_+^\nu f_\alpha(\cdot\vert t)\}(x)$  of the stable density $f_\alpha(x\vert t)$ $(0<\alpha\le1, t>0, \nu\ge0)$
with respect to a gamma  distribution $G(t \vert \mu,\lambda)$  $(\mu>0, \lambda\ge0)$
over the scale $t$ of the stable density
\begin{align}
M^\nu_{\alpha,\mu}(x\vert \lambda)
&= \int_0^\infty  \{I_+^\nu\, f_\alpha(\cdot\vert t)\}(x) \, dG(t \vert \mu,\lambda)
\label{eq:RLmixture} \\
&\equiv \frac{1}{\Gamma(\mu)}
\int_0^\infty  \{I_+^\nu\, f_\alpha(\cdot\vert t)\}(x) \, t^{\mu-1}\, e^{-\lambda t} \, dt 
\label{eq:RLmixture1}  \\
&= \frac{1}{\Gamma(\mu)}
\int_0^\infty  \{I_+^\nu\, f_\alpha\}(xt^{-1/\alpha})\, t^{\mu+(\nu-1)/\alpha-1}  \, e^{-\lambda t} \, dt 
\label{eq:RLmixture2} 
\end{align}
where the equality of $(\ref{eq:RLmixture1})$ and $(\ref{eq:RLmixture2})$ follows from Proposition~\ref{prop:RLstable}.
Setting $x=1$, it immediately follows that  
$M^\nu_{\alpha,\mu}(1\vert \lambda)$  is  the Laplace transform of a three-parameter distribution $R^\nu_{\alpha,\mu}(t)$, say,
with $\lambda$ as the Laplace transform  variable
\begin{align}
M^\nu_{\alpha,\mu}(1\vert \lambda)
&=  \int_0^\infty  e^{- \lambda t}\, dR^\nu_{\alpha,\mu}(t)  \qquad (\lambda\ge0)
\label{eq:mix3par} \\
 {\rm where}\quad 
dR^\nu_{\alpha,\mu}(t) 
 &=  \frac{1}{\Gamma(\mu)}\, \{I_+^\nu\, f_\alpha(\cdot\vert t)\}(1)\, t^{\mu-1}  \, dt
\label{eq:R3par1} \\
 &= \frac{1}{\Gamma(\mu)}\,  \{I_+^\nu\, f_\alpha\}(t^{-1/\alpha})\, t^{\mu+(\nu-1)/\alpha-1}  \, dt 
\label{eq:R3par2} 
\end{align}
Hence  $M^\nu_{\alpha,\mu}(1\vert \lambda)$  is  completely monotone.
\end{definition}

\begin{proposition}
\label{prop:RLstable1}  
Let $\nu=\beta-\alpha\gamma\ge0$ or $\beta\ge \alpha\gamma$ and $0<\alpha\le1$, $\gamma>0$.
Then Proposition~\ref{prop:RLstable} amounts to 
\begin{align}
t^\gamma \{I_+^{\beta-\alpha\gamma} f_\alpha(\cdot\vert t)\}(x) 
 &= t^{(\beta-1)/\alpha} \{I_+^{\beta-\alpha\gamma} f_\alpha\}(xt^{-1/\alpha}) 
\label{eq:RLstable3par} 
\intertext{The case $\beta-\alpha\gamma=0$ $({\it e.g.}\ \beta=\gamma=0)$ gives}
f_\alpha(x\vert t)  \equiv \{I_+^{\,0} f_\alpha(\cdot\vert t)\}(x) &= f_\alpha(xt^{-1/\alpha})t^{-1/\alpha} 
\label{eq:RLstable3par0}
\intertext{The case $\beta-\alpha\gamma=1-\alpha$ $({\it e.g.}\ \beta=\gamma=1)$ gives} 
x f_\alpha(x\vert t) = \alpha\, t \{I_+^{1-\alpha} f_\alpha(\cdot\vert t)\}(x) &= \alpha \{I_+^{1-\alpha} f_\alpha\}(xt^{-1/\alpha})
\label{eq:RLstable3par1} 
\end{align}
\end{proposition}

\begin{proof}[Proof of Proposition  \ref{prop:RLstable1}]
(\ref{eq:RLstable3par}) immediately follows from (\ref{eq:RLstable}).
In turn, (\ref{eq:RLstable3par0}) follows from (\ref{eq:RLstable3par}) for $\beta-\alpha\gamma=0$, together with  the identity
$\{I_+^{\,0} f\}(x)\equiv f(x)$ for any $f(x)$ as discussed in Section~\ref{sec:fracint}.
For $\beta-\alpha\gamma=1-\alpha$, the Laplace transform of $\{I_+^{1-\alpha} f_\alpha(\cdot\vert t)\}(x) $ is
\begin{align*} 
s^{\alpha-1} e^{-ts^\alpha} &= -\frac{1}{\alpha t}\frac{d}{ds} e^{-ts^\alpha} 
 =  \frac{1}{\alpha t} \int_0^\infty e^{-sx} x f_\alpha(x\vert t)\, dx 
\end{align*}
$ \implies\; x f_\alpha(x\vert t) = \alpha\, t \{I_+^{1-\alpha} f_\alpha(\cdot\vert t)\}(x) = \alpha \{I_+^{1-\alpha} f_\alpha\}(xt^{-1/\alpha})$,
thereby proving  (\ref{eq:RLstable3par1}).
\label{proof:RLstable1}  
\end{proof}

We may introduce yet another parameter $\theta$, say.
For any $\theta$,  $\beta+\theta-\alpha(\gamma+\theta/\alpha) = \beta-\alpha\gamma$.
Hence, if $\beta\to\beta+\theta$ and $\gamma\to\gamma+\theta/\alpha$ individually,  
the  restriction $\gamma>0$   becomes $\gamma+\theta/\alpha>0$ or $\theta>-\alpha\gamma$.
Since $\beta-\alpha\gamma$ remains unchanged,  (\ref{eq:RLstable3par}) becomes
\begin{align}
t^{\gamma+\theta/\alpha} \{I_+^{\beta-\alpha\gamma} f_\alpha(\cdot\vert t)\}(x) 
 = t^{(\beta+\theta-1)/\alpha} \{I_+^{\beta-\alpha\gamma} f_\alpha\}(xt^{-1/\alpha}) 
\label{eq:RLstable4par} 
\end{align}
which is simply (\ref{eq:RLstable3par}), scaled by $t^{\theta/\alpha}$.

For $\mu=\gamma+\theta/\alpha>0$ and $\nu=\beta-\alpha\gamma\ge0$, the mixture 
$M^{\beta-\alpha\gamma}_{\alpha,\gamma+\theta/\alpha}(x\vert \lambda)$ of Definition~\ref{def:RLmixture}
involves  three composite parameters $\{\alpha, \beta-\alpha\gamma, \gamma+\theta/\alpha\}$
constructed from 
four base parameters $\{\alpha,\beta,\gamma, \theta\}$. 
\begin{theorem}
\label{thm:ML4par}
$M^{\beta-\alpha\gamma}_{\alpha,\gamma+\theta/\alpha}(x\vert \lambda) =  \cE^{\gamma+\theta/\alpha}_{\alpha,\beta+\theta}(x\vert \lambda)
\equiv x^{\beta+\theta-1} E^{\gamma+\theta/\alpha}_{\alpha,\beta+\theta}(-\lambda x^\alpha)$ $(x\ge0)$,
 where $M^\nu_{\alpha,\mu}(x\vert \lambda)$ is the mixture of  Definition~\ref{def:RLmixture} with 
 $\mu=\gamma+\theta/\alpha>0, \nu=\beta-\alpha\gamma\ge0$.
Thus $\cE^{\gamma+\theta/\alpha}_{\alpha,\beta+\theta}(x\vert \lambda)$ 
can be expressed as a mixture of fractional integrals $\{I_+^{\beta-\alpha\gamma} f_\alpha(\cdot\vert t)\}(x)$ 
of the stable density $f_\alpha(x\vert t)$ $(t>0)$,
with a gamma mixing distribution $G(t \vert \gamma+\theta/\alpha,\lambda)$ 
for $0<\alpha\le1, \beta\ge\alpha\gamma, \theta>-\alpha\gamma$
\begin{align}
\cE^{\gamma+\theta/\alpha}_{\alpha,\beta+\theta}(x\vert \lambda) 
&\equiv x^{\beta+\theta-1} E^{\gamma+\theta/\alpha}_{\alpha,\beta+\theta}(-\lambda x^\alpha) 
\label{eq:closure}  \\
&= 
\int_0^\infty  \{I_+^{\beta-\alpha\gamma} f_\alpha(\cdot\vert t)\}(x) \, dG(t \vert \gamma+\theta/\alpha,\lambda)
\label{eq:RLvariant} \\
&\equiv \frac{1}{\Gamma(\gamma+\theta/\alpha)}
\int_0^\infty  \{I_+^{\beta-\alpha\gamma} f_\alpha(\cdot\vert t)\}(x) \, t^{\gamma+\theta/\alpha-1}\, e^{-\lambda t} \, dt 
\label{eq:RLMLmixture1}  \\
&= \frac{1}{\Gamma(\gamma+\theta/\alpha)} 
\int_0^\infty  \{I_+^{\beta-\alpha\gamma} f_\alpha\}(xt^{-1/\alpha})\, t^{(\beta+\theta-1)/\alpha-1}  \, e^{-\lambda t} \, dt 
\label{eq:RLMLmixture2} \\
(\beta-\alpha\gamma=0) \quad
&= \frac{1}{\Gamma(\gamma+\theta/\alpha)} \int_0^\infty  f_\alpha(xt^{-1/\alpha}) \, t^{(\beta+\theta-1)/\alpha-1}  \, e^{-\lambda t} \, dt  
\label{eq:RLMLmixture2a} \\
(\beta-\alpha\gamma=1-\alpha) \quad
&= \frac{1}{\Gamma(\gamma+\theta/\alpha)} \, \frac{x}{\alpha} \int_0^\infty  f_\alpha(xt^{-1/\alpha}) \, t^{(\beta+\theta-2)/\alpha-1}  \, e^{-\lambda t} \, dt 
\label{eq:RLMLmixture2b} 
\end{align}
amongst several variants of $(\ref{eq:RLMLmixture2a}), (\ref{eq:RLMLmixture2b})$ 
induced by $\beta-\alpha\gamma=0$, $\beta-\alpha\gamma=1-\alpha$ respectively.
\end{theorem}
\begin{proof}[Proof of Theorem \ref{thm:ML4par}]
\label{proof:ML4par}
The Laplace transform  of (\ref{eq:RLMLmixture1}) is
\begin{align*}
\frac{s^{\alpha\gamma-\beta}}{\Gamma(\gamma+\theta/\alpha)} \int_0^\infty 
\, t^{\gamma+\theta/\alpha-1} e^{-(\lambda+s^\alpha) t} \,dt  
&= \frac{s^{\alpha\gamma-\beta}}{(\lambda+s^\alpha)^{\gamma+\theta/\alpha}} 
=  \frac{s^{\alpha(\gamma+\theta/\alpha)-(\beta+\theta)}}{(\lambda+s^\alpha)^{\gamma+\theta/\alpha}} 
\end{align*}
By (\ref{eq:LaplaceML3par}), the rightmost expression  is the Laplace transform of 
$\cE^{\gamma+\theta/\alpha}_{\alpha,\beta+\theta}(x\vert \lambda)$.
The equivalence of (\ref{eq:RLMLmixture1}) and (\ref{eq:RLMLmixture2})  follows from~(\ref{eq:RLstable4par}).
Hence the equivalence of $(\ref{eq:closure})$, $(\ref{eq:RLMLmixture1})$ and $(\ref{eq:RLMLmixture2})$.
For $\beta-\alpha\gamma=0$, using $(\ref{eq:RLstable3par0})$ in $(\ref{eq:RLMLmixture2})$ leads to $(\ref{eq:RLMLmixture2a})$.
For $\beta-\alpha\gamma=1-\alpha$, using $(\ref{eq:RLstable3par1})$ in $(\ref{eq:RLMLmixture2})$ leads to $(\ref{eq:RLMLmixture2b})$.
\end{proof}

\begin{corollary}
\label{cor:ML4par}
The Prabhakar function 
$\cE^{\gamma+\theta/\alpha}_{\alpha,\beta+\theta}(1\vert \lambda) \equiv E^{\gamma+\theta/\alpha}_{\alpha,\beta+\theta}(-\lambda)$  
is  completely monotone.
Equivalently,  
 it is the Laplace transform of a four-parameter  distribution $Q^{\gamma}_{\alpha,\beta,\theta}(t)$
\begin{align}
E^{\gamma+\theta/\alpha}_{\alpha,\beta+\theta}(-\lambda)
&= \int_0^\infty  e^{- \lambda t}\, dQ^{\gamma}_{\alpha,\beta,\theta}(t)  \qquad (\lambda\ge0)
\label{eq:ML4parQ}  \\
 {\rm where}\quad 
dQ^{\gamma}_{\alpha,\beta,\theta}(t) 
 &=  \frac{1}{\Gamma(\gamma+\theta/\alpha)}\, \{I_+^{\beta-\alpha\gamma} f_\alpha(\cdot\vert t)\}(1)\, t^{\gamma+\theta/\alpha-1}  \, dt
\label{eq:Q4par1} \\
&= \frac{1}{\Gamma(\gamma+\theta/\alpha)}\,  \{I_+^{\beta-\alpha\gamma} f_\alpha\}(t^{-1/\alpha})\, t^{(\beta+\theta-1)/\alpha-1}  \, dt 
\label{eq:Q4par2}  \\
(\beta-\alpha\gamma=0) \quad
&= \frac{1}{\Gamma(\gamma+\theta/\alpha)} \,  f_\alpha(t^{-1/\alpha})\, t^{(\beta+\theta-1)/\alpha-1}  \, dt 
\label{eq:Q4par2a} \\
(\beta-\alpha\gamma=1-\alpha) \quad
&= \frac{1}{\Gamma(\gamma+\theta/\alpha)} \, \frac{1}{\alpha}\, f_\alpha(t^{-1/\alpha}) \, t^{(\beta+\theta-2)/\alpha-1} \, dt 
\label{eq:Q4par2b} 
\end{align}
amongst several variants of $(\ref{eq:Q4par2a}), (\ref{eq:Q4par2b})$ 
induced by $\beta-\alpha\gamma=0$, $\beta-\alpha\gamma=1-\alpha$ respectively.
\end{corollary}

\begin{proof}[Proof of Corollary \ref{cor:ML4par}]
\label{proof:corML4par}
(\ref{eq:ML4parQ}) 
follows  from Theorem~\ref{thm:ML4par} by setting $x=1$. 
For $\beta-\alpha\gamma=0$, using $(\ref{eq:RLstable3par0})$ in $(\ref{eq:Q4par2})$ leads to $(\ref{eq:Q4par2a})$.
For $\beta-\alpha\gamma=1-\alpha$, using $(\ref{eq:RLstable3par1})$ in $(\ref{eq:Q4par2})$ leads to $(\ref{eq:Q4par2b})$.
\end{proof}

\begin{corollary}
\label{cor:ML4parMoments}
Setting $\lambda =0$ in $(\ref{eq:ML4parQ})$ and reading 
$E^{\gamma+\theta/\alpha}_{\alpha,\beta+\theta}(0)=1/\Gamma(\beta+\theta)$ from ~$(\ref{eq:ML3parseries})$ 
\begin{align*}
 \int_0^\infty  dQ^{\gamma}_{\alpha,\beta,\theta}(t) 
&= E^{\gamma+\theta/\alpha}_{\alpha,\beta+\theta}(0)  = \frac{1}{\Gamma(\beta+\theta)}
\end{align*}
Hence $P^{\gamma}_{\alpha,\beta,\theta}(t) \equiv \Gamma(\beta+\theta) Q^{\gamma}_{\alpha,\beta,\theta}(t)$
is a  probability distribution  and its  Laplace transform is
$\Gamma(\beta+\theta)E^{\gamma+\theta/\alpha}_{\alpha,\beta+\theta}(-\lambda)$.
For  $n\ge0$, the moments  of $P^{\gamma}_{\alpha,\beta,\theta}(t)$ are
\begin{align}
\int_0^\infty  t^n \, dP^{\gamma}_{\alpha,\beta,\theta}(t)
&= \frac{\Gamma(\beta+\theta)\,\Gamma(\gamma+n+\theta/\alpha)}{\Gamma(\gamma+\theta/\alpha)\,\Gamma(\beta+\alpha n+\theta)}
\label{eq:ML4parMoments} 
\intertext{More generally, for $\lambda\ge0$ and $q>-\gamma-\theta/\alpha$} 
 \int_0^\infty  e^{- \lambda t}\, t^q \, dP^{\gamma}_{\alpha,\beta,\theta}(t) 
&= \frac{\Gamma(\beta+\theta)\,\Gamma(\gamma+q+\theta/\alpha)}{\Gamma(\gamma+\theta/\alpha)}\, 
E^{\gamma+q+\theta/\alpha}_{\alpha,\beta+\alpha q+\theta}(-\lambda)  \ge 0
\label{eq:ML4parCM} 
\end{align}
\end{corollary}

\begin{proof}[Proof of Corollary \ref{cor:ML4parMoments}]
\label{proof:ML4parMoments}
By (\ref{eq:Q4par1})   together with $P^{\gamma}_{\alpha,\beta,\theta}(t) \equiv \Gamma(\beta+\theta) Q^{\gamma}_{\alpha,\beta,\theta}(t)$
\begin{align*}
 t^n\, dP^{\gamma}_{\alpha,\beta,\theta}(t)  
 &=  \frac{\Gamma(\beta+\theta)}{\Gamma(\gamma+\theta/\alpha)}\, 
 \{I_+^{\beta-\alpha\gamma} f_\alpha(\cdot\vert t)\}(1)\, t^{\gamma+(\alpha n+\theta)/\alpha-1}  \, dt \\
 &= \frac{\Gamma(\beta+\theta)\Gamma(\gamma+n+\theta/\alpha)}{\Gamma(\gamma+\theta/\alpha)\Gamma(\beta+\alpha n+\theta)}\, 
 dP^{\gamma}_{\alpha,\beta,\alpha n+\theta}(t) \\
\implies
\int_0^\infty  t^n \, dP^{\gamma}_{\alpha,\beta,\theta}(t)
&= \frac{\Gamma(\beta+\theta)\,\Gamma(\gamma+n+\theta/\alpha)}{\Gamma(\gamma+\theta/\alpha)\,\Gamma(\beta+\alpha n+\theta)}
\end{align*}
since  $P^{\gamma}_{\alpha,\beta,\alpha n+\theta}(t)$ is normalised.
For the more general case with $\lambda\ge0$ and $q>-\gamma-\theta/\alpha$
\begin{align*}
 \int_0^\infty  e^{- \lambda t}\, t^q \, dP^{\gamma}_{\alpha,\beta,\theta}(t) 
 &= \frac{\Gamma(\beta+\theta)\Gamma(\gamma+q+\theta/\alpha)}{\Gamma(\gamma+\theta/\alpha)\Gamma(\beta+\alpha q+\theta)}\, 
 \int_0^\infty  e^{- \lambda t}\, dP^{\gamma}_{\alpha,\beta,\alpha q+\theta}(t) \\
&= \frac{\Gamma(\beta+\theta)\,\Gamma(\gamma+q+\theta/\alpha)}{\Gamma(\gamma+\theta/\alpha)}\, 
E^{\gamma+q+\theta/\alpha}_{\alpha,\beta+\alpha q+\theta}(-\lambda)  \ge 0 
\end{align*}
since  the Laplace  transform of $P^{\gamma}_{\alpha,\beta,\alpha q+\theta}(t)$ is 
$\Gamma(\beta+\alpha q+\theta)E^{\gamma+q+\theta/\alpha}_{\alpha,\beta+\alpha q+\theta}(-\lambda)$
\end{proof}


Our argument has solely appealed to probability  theory and fractional integration, 
with the Riemann-Liouville fractional integral of the conditional stable density as the basic construct.
To borrow a term from number theory, we might describe the argument as ``elementary"   in the sense that it has not made 
explicit reference to complex analysis.

The special cases of Theorem~\ref{thm:ML4par} studied in the literature follow two distinct  approaches:

\paragraph{Complex Analytic:} 
The starting point in this context  is the work by Pollard~\cite{PollardML} on the completely monotone character of the one-parameter 
Mittag-Leffler function $E_\alpha(-x)$,
deriving a distribution  $P_\alpha(t)$ whose Laplace transform is $E_\alpha(-x)$.
G\'{o}rska {\it et al.}~\cite{Gorska} generalised Pollard's complex analytic approach to prove  the 
completely monotone character of the Prabhakar  function $E^\gamma_{\alpha,\beta}(-x)$.
There is  well-established  work  (not discussed in this paper) that also appeals to complex analytic methods to prove 
the completely monotone character of  $\cE^\gamma_{\alpha,\beta}(x\vert \lambda)$,
{\it i.e.}\ that there exists  a distribution  whose Laplace  transform is $\cE^\gamma_{\alpha,\beta}(x\vert \lambda)$
(deOliveira {\it et al.}~\cite{Oliveira}, Mainardi and Garrappa~\cite{MainardiGarrappa}, Tomovski {\it et al.}~\cite{Tomovski}).
In a nutshell, the point of departure of this literature is complex analysis and, in the parametric  framework  of this paper, 
it is characterised by $\theta=0$.

\paragraph{Probabilistic:}
In the probability literature, 
$P_\alpha(t)$ is known as the Mittag-Leffler distribution
(one of two distributions bearing the  name).  
The generalisation to  two parameters ($\alpha,\theta$), referred to as the generalised  Mittag-Leffler distribution,
is the case $P_{\alpha,\theta}(t)\equiv P^{\gamma=1}_{\alpha,\beta=1,\theta}(t)$ of this paper.

We proceed to sketch  both complex  analytic and probabilistic  approaches in the literature.
A fundamental  objective of this paper is a unifying probabilistic framework for 
these disjoint bodies of literature.

\section{The Case $\theta=0$}

\subsection{$\{\beta=\gamma=1\}$}
\label{sec:1par}
$\cE_\alpha(x\vert \lambda)\equiv \cE^{1}_{\alpha,1,0}(x\vert \lambda) = 
 E^{1}_{\alpha,1}(-\lambda x^\alpha) \equiv E_\alpha(-\lambda x^\alpha)$.
Hence (\ref{eq:RLMLmixture2b}) of Theorem~\ref{thm:ML4par} becomes
\begin{align}
\cE_\alpha(x\vert \lambda) =  E_\alpha(-\lambda x^\alpha) 
&=  \frac{x}{\alpha} \int_0^\infty  f_\alpha(x\vert t) \, t^{-1}  \, e^{-\lambda t} \, dt  \nonumber \\
&=  \frac{x}{\alpha} \int_0^\infty  f_\alpha(xt^{-1/\alpha}) \, t^{-1/\alpha-1}  \, e^{-\lambda t} \, dt 
\label{eq:MLmixture} 
\end{align}

Setting $x=1$, the  Mittag-Leffler function   $E_\alpha(-\lambda)$ is completely monotone by virtue of being 
the Laplace transform of the  distribution  $P_\alpha(t) \equiv P^{1}_{\alpha,1,0}(t)$, {\it i.e.}\
(\ref{eq:MLmixture}) becomes
\begin{align}
\cE_\alpha(1\vert \lambda)  = E_\alpha(-\lambda)
&=  \int_0^\infty  e^{- \lambda t}\, dP_\alpha(t)  \qquad (\lambda\ge0)
\label{eq:ML1par} \\
 {\rm where}\quad 
dP_\alpha(t) 
 &=   \frac{1}{\alpha} \,  f_\alpha(1\vert t) \, t^{-1} \, dt 
   =   \frac{1}{\alpha} \,  f_\alpha(t^{-1/\alpha}) \, t^{-1/\alpha-1} \, dt 
 \label{eq:P1par} 
 \end{align}
$P_\alpha(t)$  may be written in terms of the stable distribution $F_\alpha$: 
\begin{align}
P_\alpha(t) &= \ \frac{1}{\alpha} \int_0^t  f_\alpha(u^{-1/\alpha})\, u^{-1/\alpha-1} \, du 
= \int_{t^{-1/\alpha}}^\infty  f_\alpha(y)\, dy 
\nonumber \\
&=  1- \int_0^{t^{-1/\alpha}} f_\alpha(y)\, dy   
  \equiv 1-  F_\alpha(t^{-1/\alpha}) = 1-  F_\alpha(1\vert t)
\label{eq:PollardDistribution}
\end{align}

Adopting a  complex analytic approach, Pollard used  the contour integral representation:
\begin{align}
 E_\alpha(-x) &= \frac{1}{2\pi i}\oint_{C} \frac{s^{\alpha-1}e^s}{x+s^{\alpha}}\,ds  
                       = \frac{1}{2\pi i\alpha}\oint_{C^\prime} \frac{e^{z^{\frac{1}{\alpha}}}}{x+z}\,dz
\label{eq:MLcontour}
\end{align}
to prove that the Mittag-Leffler function $E_\alpha(-x)$ is the Laplace  transform   of $P_\alpha(t)$ given by~(\ref{eq:P1par}).
We  may thus refer to $P_\alpha(t)$ as the  Pollard distribution.

 Interestingly, Pollard~\cite{PollardML}  led with an  opening remark about a probabilistic  approach due to Feller before turning to his own 
 complex analytic approach:
\begin{quote}
``W.~Feller communicated to me his discovery -- by the methods of probability theory -- that if $0\le \alpha \le1$ 
the function $E_\alpha(-x)$ is completely monotonic for $x\ge0$. 
This means that it can be written in the form
\begin{align*}
E_\alpha(-x) &= \int_{0}^\infty e^{-xt} dP_\alpha(t) 
\end{align*}
where $P_\alpha(t)$ is nondecreasing and bounded.
In this note we shall prove this fact directly  and determine the function $P_\alpha(t)$ explicitly.'' \newline
 [we use $P_\alpha$ where Pollard used $F_\alpha$, having  used the latter  for the stable distribution]
\end{quote}

In his  method of probability theory, Feller~\cite{Feller2}~(XIII.8) considered   the two-dimensional Laplace transform of
 $1- F_\alpha(x\vert t)  \equiv 1- F_\alpha(xt^{-1/\alpha})$ to prove the  result. 
This paper is in the spirit of Feller's probabilistic reasoning  rather than Pollard's complex analysis.

\subsection{$\{\beta\ge\alpha\gamma,\gamma>0\}$}
\label{sec:3par}
$\cE^{\gamma}_{\alpha,\beta}(x\vert \lambda) \equiv \cE^{\gamma}_{\alpha,\beta,0}(x\vert \lambda)  
= x^{\beta-1} E^{\gamma}_{\alpha,\beta}(-\lambda x^\alpha)$.
Hence (\ref{eq:RLMLmixture2})  of Theorem~\ref{thm:ML4par} becomes
\begin{align}
\cE^{\gamma}_{\alpha,\beta}(x\vert \lambda) 
&\equiv x^{\beta-1} E^{\gamma}_{\alpha,\beta}(-\lambda x^\alpha)   \nonumber \\
&= \frac{1}{\Gamma(\gamma)}
\int_0^\infty  \{I_+^{\beta-\alpha\gamma} f_\alpha(\cdot\vert t)\}(x) \, t^{\gamma-1}\, e^{-\lambda t} \, dt 
\label{eq:RLMLmix3par1}  \\
&= \frac{1}{\Gamma(\gamma)} 
\int_0^\infty  \{I_+^{\beta-\alpha\gamma} f_\alpha\}(xt^{-1/\alpha})\, t^{(\beta-1)/\alpha-1}  \, e^{-\lambda t} \, dt 
\label{eq:RLMLmix3par2} 
\end{align}

By Corollary~\ref{cor:ML4par}, the Prabhakar function $\cE^{\gamma}_{\alpha,\beta}(1\vert \lambda) = E^{\gamma}_{\alpha,\beta}(-\lambda)$
is completely  monotone by virtue of being the Laplace transform of 
$Q^{\gamma}_{\alpha,\beta}(t) \equiv Q^{\gamma}_{\alpha,\beta,\theta=0}(t)$.  
Equivalently,  $\Gamma(\beta)\cE^{\gamma}_{\alpha,\beta}(1\vert \lambda) = \Gamma(\beta)E^{\gamma}_{\alpha,\beta}(-\lambda)$ 
is the Laplace transform of the probability distribution $P^{\gamma}_{\alpha,\beta}(t)$ given by 
\begin{align}
dP^{\gamma}_{\alpha,\beta}(t) 
 &=  \frac{\Gamma(\beta)}{\Gamma(\gamma)}\, \{I_+^{\beta-\alpha\gamma} f_\alpha(\cdot\vert t)\}(1)\, t^{\gamma-1}  \, dt \\
 &= \frac{\Gamma(\beta)}{\Gamma(\gamma)}\,  \{I_+^{\beta-\alpha\gamma} f_\alpha\}(t^{-1/\alpha})\, t^{(\beta-1)/\alpha-1}  \, dt  \\
\int_0^\infty  t^n \, dP^{\gamma}_{\alpha,\beta}(t)
&= \frac{\Gamma(\beta)\,\Gamma(\gamma+n)}{\Gamma(\gamma)\,\Gamma(\beta+\alpha n)} 
\end{align}
The  representation (\ref{eq:RLMLmix3par2})  for $x=1$ is technically equivalent to 
equation (2.4) in G\'{o}rska {\it et al.}~\cite{Gorska}.
Inspired by  Pollard~\cite{PollardML}, G\'{o}rska {\it et al.}~\cite{Gorska} took Laplace  inversion as the point of departure:
\begin{align}
E^\gamma_{\alpha,\beta}(-\lambda x^\alpha) 
&= \frac{x^{1-\beta}}{2\pi i}\oint_{C} e^{s x\, }\frac{s^{\alpha\gamma-\beta}}{(\lambda+s^{\alpha})^\gamma}\,ds  
\label{eq:MLcontour}
\end{align}

The balance of G\'{o}rska {\it et al.}~\cite{Gorska}  is devoted to  finding an explicit formula for a function $f^\gamma_{\alpha,\beta}(x)$ featuring in
the paper in terms of the Meijer $G$ function and associated confluent Wright function.
Equation (22) in Tomovski {\it et al.}~\cite{Tomovski}, similarly  derived through contour integration, also involves the Wright function.
As defined by equation~(2.6)   in G\'{o}rska {\it et al.}~\cite{Gorska}, $f^\gamma_{\alpha,\beta}(x)$  is identical to 
\begin{align} 
\{I_+^{\beta-\alpha\gamma} f_\alpha\}(x) 
   &= \frac{1}{\Gamma(\beta-\alpha\gamma)}  \int_0^x  (x-u)^{\beta-\alpha\gamma-1}  f_\alpha(u)\, du 
 \label{eq:rhostarstablef}
\end{align}
We are content to leave  this in the  simple fractional integral form rather than express it in terms of  special functions.
In  any event, we have  actually worked with the conditional   density 
\begin{align*}
t^\gamma \{I_+^{\beta-\alpha\gamma} f_\alpha(\cdot\vert t)\}(x) 
 &= t^{(\beta-1)/\alpha} \{I_+^{\beta-\alpha\gamma} f_\alpha\}(xt^{-1/\alpha}) 
\end{align*}
where we assigned a gamma prior distribution  to the scale parameter $t$.
This conditional density  reduces to (\ref{eq:rhostarstablef}) only for the particular choice $t=1$.

\section{The  Case $\theta\ne0$}
\label{sec:2par}
The instances studied in the probability literature amount to the choice 
$\beta=\gamma=1$ (a particular case  of $\beta-\alpha\gamma=1-\alpha$)
and, rather indirectly, $\beta=\gamma=0$ (a particular case  of $\beta-\alpha\gamma=0$).
Starting with $\beta=\gamma=1$, (\ref{eq:RLMLmixture2b}) in Theorem~\ref{thm:ML4par} gives
\begin{align}
\cE^{1+\theta/\alpha}_{\alpha,1+\theta}(x\vert \lambda) 
&= x^{\theta} E^{1+\theta/\alpha}_{\alpha,1+\theta}(-\lambda x^\alpha) 
\nonumber \\
&= \frac{1}{\Gamma(1+\theta/\alpha)} \frac{x}{\alpha}
\int_0^\infty  f_\alpha(x\vert t) \, t^{\theta/\alpha-1} \, e^{-\lambda t} \, dt 
\label{eq:RLMLmix2par1}  \\
&= \frac{1}{\Gamma(1+\theta/\alpha)} \frac{x}{\alpha}
\int_0^\infty  f_\alpha(x t^{-1/\alpha}) \, t^{(\theta-1)/\alpha-1} \, e^{-\lambda t} \, dt 
\label{eq:RLMLmix2par2} 
\end{align}
$\cE^{1+\theta/\alpha}_{\alpha,1+\theta}(1\vert \lambda) = E^{1+\theta/\alpha}_{\alpha,1+\theta}(-\lambda)$ 
is completely monotone by virtue of being 
the Laplace transform of a two-parameter distribution  $Q_{\alpha,\theta}(t) \equiv Q^{1}_{\alpha,1,\theta}(t)$.
Equivalently,  
$\Gamma(1+\theta)E^{1+\theta/\alpha}_{\alpha,1+\theta}(-\lambda)$ 
is the Laplace transform of the probability distribution $P_{\alpha,\theta}(t)\equiv P^{1}_{\alpha,1,\theta}(t)$ 
\begin{align}
\Gamma(1+\theta) E^{1+\theta/\alpha}_{\alpha,1+\theta}(-\lambda)
&=  \int_0^\infty  e^{- \lambda t}\, dP_{\alpha,\theta}(t)  \quad (\lambda\ge0)
\label{eq:ML2par} \\
 {\rm where}\quad 
dP_{\alpha,\theta}(t) 
 &=  \frac{\Gamma(1+\theta)}{\Gamma(1+\theta/\alpha)}\,  \frac{1}{\alpha} \,   f_\alpha(t^{-1/\alpha}) \, t^{(\theta-1)/\alpha-1}  \, dt 
 \label{eq:P2par2}   \\
 &=  \frac{\Gamma(1+\theta) }{\Gamma(1+\theta/\alpha)}\,   t^{\theta/\alpha} \, dP_\alpha(t) \quad \theta>-\alpha
\label{eq:P2par3} \\
 {\rm and}\quad 
\int_0^\infty  t^n \, dP_{\alpha,\theta}(t)
&= \frac{\Gamma(1+\theta)\,\Gamma(1+n+\theta/\alpha)}{\Gamma(1+\theta/\alpha)\,\Gamma(1+\alpha n+\theta)} 
\label{eq:P2moments}
\intertext{Since $\Gamma(1+\theta)=\theta\Gamma(\theta)$ and  $\Gamma(1+\theta/\alpha) = (\theta/\alpha)\Gamma(\theta/\alpha)$ {\it etc.},  
(\ref{eq:P2par2}) and (\ref{eq:P2moments}) may be written as} 
dP_{\alpha,\theta}(t) 
 &=  \frac{\Gamma(\theta)}{\Gamma(\theta/\alpha)} \,   f_\alpha(t^{-1/\alpha}) \, t^{(\theta-1)/\alpha-1}  \, dt 
 \label{eq:P2par4}  \\
 \int_0^\infty  t^n \, dP_{\alpha,\theta}(t)
&= \frac{\Gamma(\theta)\,\Gamma(n+\theta/\alpha)}{\Gamma(\theta/\alpha)\,\Gamma(\alpha n+\theta)} 
\label{eq:P2moments1}
\end{align}
but this obscures the original range $\theta>-\alpha$, suggesting instead that $\theta>0$.
In fact, (\ref{eq:P2par4}) coincides with an alternative definition $P_{\alpha,\theta}(t)\equiv P^{0}_{\alpha,0,\theta}(t)$, 
given by~(\ref{eq:Q4par2a}) (times $\Gamma(\theta)$) for $\beta=\gamma=0$,
with  Laplace transform  $\Gamma(\theta)E^{\theta/\alpha}_{\alpha,\theta}(-\lambda)$. 
We  regard $P^{1}_{\alpha,1,\theta}(t)$ and $P^{0}_{\alpha,0,\theta}(t)$  as logically distinct  special cases of 
$P^{\gamma}_{\alpha,\beta,\theta}(t)$, both compatible  with $\theta>-\alpha\gamma$.
Henceforth we take $P_{\alpha,\theta}(t)$ to be  $P^{1}_{\alpha,1,\theta}(t)$.

The Pollard distribution $P_\alpha(t)$ is known as  the Mittag-Leffler distribution in the probability literature.
 It is one of  two distributions bearing the  name: $P_\alpha(t)= 1- F_\alpha(1\vert t)=1- F_\alpha(t^{-1/\alpha})$ and $1-E_\alpha(-t^\alpha)$ 
 (the latter was so named by Pillai~\cite{Pillai}).
$P_{\alpha,\theta}(t)$  is known in the same literature 
as the generalised Mittag-Leffler distribution (Pitman~\cite{Pitman_CSP}, p70 (3.27)).
It is also denoted by ${\rm ML(\alpha,\theta})$ 
(Goldschmidt and Haas~\cite{GoldschmidtHaas}, Ho {\it  et al.}~\cite{HoJamesLau}). 
The derivation in the probability  literature takes the form of  a limiting distribution of  a  P{\' o}lya urn scheme 
(Janson~\cite{Janson}).
As concisely  described  in~\cite{GoldschmidtHaas}, 
the same limiting distribution arises from the Chinese restaurant process, which is 
closely  related to the Poisson-Dirichlet process ${\rm PD(\alpha,\theta})$
of Pitman and Yor~\cite{PitmanYor}.

\section{Infinitely Divisible Distributions}
\label{sec:ID}

There is  an   intimate relationship between completely monotone functions and the theory of 
infinitely divisible  distributions on the nonnegative half-line $\mathbb{R}_{+}=[0,\infty)$ 
(Feller~\cite{Feller2} (XIII.4,~XIII.7), Steutel and van Harn~\cite{SteutelvanHarn} (III)).
Sato~\cite{Sato}  considers infinitely divisible distributions on $\mathbb{R}^d$,
but the deliberate restriction to $\mathbb{R}_{+}$ makes for simpler discussion 
and relates directly to the core concept of complete monotonicity that is of interest here.
There is also an intimate  link to the generalised gamma convolutions studied by Bondesson~\cite{Bondesson}.

A distribution on $\mathbb{R}_{+}$ with density $f(x\vert t)$  ($t>0$) is infinitely divisible iff its Laplace transform takes the form
$\widetilde f(s\vert t)=e^{-t\psi(s)}$ ($s>0$) where $\psi(s)$  is a positive function and its derivative $\psi\,^{\prime}(s)$ is completely monotone.
Hence there exists  a density $\rho(x)$ with Laplace transform $\widetilde \rho(s)= \psi\,^{\prime}(s)$:
\begin{align}
 -\frac{\widetilde f\,^\prime(s\vert t)} {\widetilde f(s\vert t)} &=  t\, \psi\,^{\prime}(s) =  t \, \widetilde \rho(s)  \nonumber \\
\text{or}\;  -\widetilde f\,^\prime(s\vert t)  &= t \, \widetilde \rho(s)\times \widetilde f(s\vert t) \
 \label{eq:IDLaplace} \\
\implies x f(x\vert t)  &= t \, \{\rho\star f(\cdot\vert t)\}(x)
\label{eq:IDConv}
\end{align}
These are equivalent ways to characterise infinitely divisible distributions.

For the stable   density $f_\alpha(x\vert t)$ ($0<\alpha< 1$), $\psi_\alpha(s)=s^\alpha$ and
$\psi_\alpha\,^{\prime}(s)=\rho_\alpha(s)=\alpha s^{\alpha-1}$ $\implies$ $\rho_\alpha(x) = \alpha\,x^{-\alpha}/\Gamma(1-\alpha)$, so that 
$\{\rho_\alpha\star f_\alpha(\cdot\vert t)\}(x) \equiv \alpha\, \{I_+^{1-\alpha} f_\alpha(\cdot\vert t)\}(x)$.
Hence, for the stable case, (\ref{eq:IDConv}) reproduces~(\ref{eq:RLstable3par1})  from the perspective of  infinite divisibility.
Put differently, the stable  distribution may be looked upon as a bridge between the theory of infinitely divisible distributions and 
Riemann-Liouville fractional integration.
(\ref{eq:IDConv}) holds for all infinitely divisible densities, however the convolution $\{\rho\star f(\cdot\vert t)\}(x)$ need not be expressible as
a fractional integral.

\section{Discussion}
\label{sec:discussion}

It is worth reviewing   the mathematical concepts and theorem(s) that arise in this paper:
\begin{enumerate}
\item $F_\alpha(x\vert t)$, the one-sided stable  distribution  with density $f_\alpha(x\vert t)$.
\item $\{I_+^\nu\, f_\alpha(\cdot\vert t)\}(x)$, the Riemann-Liouville fractional integral of $f_\alpha(x\vert t)$.
\item $G(t\vert \mu,\lambda)$, the gamma distribution (with shape and scale  parameters $\mu,\lambda$) assigned to 
the scale parameter $t$ of the stable distribution.
\item $M^\nu_{\alpha,\mu}(x\vert \lambda)$, the mixture of $\{I_+^\nu\, f_\alpha(\cdot\vert t)\}(x)$ with respect to $G(t\vert \mu,\lambda)$.
\item Theorem~\ref{thm:ML4par} states that 
$M^{\beta-\alpha\gamma}_{\alpha,\gamma+\theta/\alpha}(x\vert \lambda) =  \cE^{\gamma+\theta/\alpha}_{\alpha,\beta+\theta}(x\vert \lambda)
\equiv x^{\beta+\theta-1} E^{\gamma+\theta/\alpha}_{\alpha,\beta+\theta}(-\lambda x^\alpha)$.
\item $\Gamma(\beta+\theta)\, \cE^{\gamma+\theta/\alpha}_{\alpha,\beta+\theta}(1\vert \lambda) \equiv 
\Gamma(\beta+\theta)\,E^{\gamma+\theta/\alpha}_{\alpha,\beta+\theta}(-\lambda)$
is the Laplace transform  of a probability distribution $P^{\gamma}_{\alpha,\beta,\theta}(t)$
that we might refer to as the four-parameter Pollard distribution.
\end{enumerate}
This combination of stable distribution, fractional   integrals and gamma distribution  is a novel probabilistic  approach to the  study of the 
Prabhakar function and its completely monotone  character.
Unlike approaches inspired by Pollard's derivation of $P_\alpha(t)$, it does not invoke complex analysis as a starting point.
Equally, the approach of this paper takes a different probabilistic  route from that of  random experiment metaphors 
(P\'{o}lya urns, Chinese restaurant process). 
It also does not invoke    other concepts  associated with the construction  of the generalised Mittag-Leffler distribution 
$P_{\alpha,\theta}(t)\equiv P^{1}_{\alpha,1,\theta}(t)$,
 such as polynomial tilting of the stable density $f_\alpha(x)\to f_{\alpha,\theta}(x)\propto x^{-\theta}f_\alpha(x)$
 (Arbel {\it et al.}~\cite{Arbel}, Devroye~\cite{Devroye}, James~\cite{James_Lamperti}).

We hasten to add that the contrasting of different  approaches is not meant to be a comment on  their relative merit.
Diversity  of approach is commonplace in  probability theory. 
For example, in a  context of nonparametric  Bayesian analysis, 
 Ferguson~\cite{Ferguson1} constructed the Dirichlet process based on 
the  gamma distribution as the fundamental probabilistic concept, 
without invoking the metaphor of a random experiment.
In a nutshell, the Dirichlet distribution arises from  normalising gamma-distributed   variables over  disjoint 
``cells" of a  partition or tiling of  some domain such as  time or space
(formally, both the variable and the gamma shape parameter are measures over the domain and thus additive under merging of disjoint sub-domains).
By Condition~C in Ferguson~\cite{Ferguson1}, the Dirichlet process is  the set of mutually consistent Dirichlet distributions imposed by
the persistence of  gamma distribution behaviour on  individual cells under  merging or splitting  of cells  of any partition
(to an arbitrarily fine limit).

Blackwell and MacQueen~\cite{BlackwellMacQueen}  observed that the Ferguson approach
``involves a rather deep study of the gamma process'' as they proceeded to give  an alternate construction based on 
a  generalised P{\' o}lya urn scheme.
Such metaphors (P{\' o}lya urn draws, stick-breaking, Chinese restaurant table selection) have  become  the popular motivation for  
the Dirichlet process and the  Poisson-Dirichlet  process ${\rm PD(\alpha,\theta})$
used to model discrete systems such as random partitions of the integers.
Given $n$ integers, the number $K_n$ of partition blocks  is analogous to the number of  tables occupied by the 
first $n$ customers for the Chinese restaurant process.
The variable $\lim_{n\to\infty}n^{-\alpha}K_n$  has the generalised Mittag-Leffler distribution  ${\rm ML(\alpha,\theta})$,
also known as the $\alpha$-diversity of ${\rm PD(\alpha,\theta})$.
Pursuing the restaurant metaphor, M\"{o}hle~\cite{Mohle} added a cocktail bar  in order to construct a
three-parameter Mittag-Leffler distribution  ${\rm ML(\alpha,\beta,\gamma})$, 
referencing  Tomovski {\it et al.}~\cite{Tomovski} for its explicit functional  form 
(G\'{o}rska {\it et al.}~\cite{Gorska}  and Tomovski {\it et al.}~\cite{Tomovski} followed a complex analytic derivation as discussed in Section~\ref{sec:3par}).
${\rm ML(\alpha,\beta,\gamma})$ is equivalent to $P^{\gamma}_{\alpha,\beta}(t)$ of this paper.


Akin to Ferguson and in keeping with Feller's  proof that $E_\alpha(-x)$ is the Laplace transform  of $P_\alpha(t)\equiv 1-F_\alpha(1\vert t)$,
we have chosen  to work directly with distributions of continuous variables  on $[0,\infty)$
to construct $P^{\gamma}_{\alpha,\beta,\theta}(t)$, which subsumes the generalised Mittag-Leffler distribution  $P_{\alpha,\theta}(t)$
as the special  case $\beta=\gamma=1$.
It is worth exploring  more general choices of  ($\beta,\gamma$), 
which may lead to    four-parameter analogues  of the  two-parameter Poisson-Dirichlet distribution.
 

We note that we need not have made specific reference to  fractional integrals. 
It would have sufficed to refer solely to convolutions of distributions, given that the  Riemann-Liouville fractional integral  
$\{I_+^\nu\, f_\alpha(\cdot\vert t)\}(x) \equiv \{h_\nu\star f_\alpha(\cdot\vert t)\}(x)$  where $h_\nu(x)=x^{\nu-1}/\Gamma(\nu)$
is a particular instance of the Laplace  convolution of densities.
While this might have made for an exclusively  probabilistic  narrative, Mittag-Leffler functions are so intimately associated  with fractional calculus 
that it seemed logical to 
import the language of fractional integrals into our probabilistic reasoning.
Finally, it is well worth exploring choices of mixing distribution other than the gamma distribution.
For example, a Poisson distribution  would induce a discrete sum over   weighted  fractional integrals.
This is but one avenue of future exploration.


 \bibliography{MittagLeffler.bib}{}
\bibliographystyle{plain} 
\end{document}